\documentclass[letterpaper, 11pt]{article}
\usepackage{amsmath,amssymb, amsthm, graphicx, setspace}
\usepackage[hmargin=1.5in,vmargin=1.5in]{geometry}

\newtheorem{thm}{Theorem}

\newtheorem{trig}{Trig. Identity}

\newtheorem{lemma}[thm]{Lemma}
\newtheorem*{thm*}{Theorem}

\theoremstyle{example}{

\newtheorem{remark}[thm]{Remark}

\newtheorem{?}{Question}
}

\DeclareMathOperator{\conv}{conv}

\newcommand{\bS}{\mathbb{S}}
\newcommand{\R}{\mathbb{R}}

\newcommand{\N}{\mathbb{N}}
\newcommand{\Z}{\mathbb{Z}}

\title{Edges of the Barvinok-Novik orbitope}
\author{Cynthia Vinzant \footnote{University of California, Berkeley, Department of Mathematics, cvinzant@math.berkeley.edu}}

\date{\today}
\begin{document}

\maketitle

\begin{abstract}
Here we study the $k$th symmetric trigonometric moment curve and its convex hull, the Barvinok-Novik orbitope. 
In 2008, Barvinok and Novik introduce these objects and show that there is some threshold so that for two points on $\bS^1$ with arclength below this threshold the line segment between their lifts to the curve form an edge on the Barvinok-Novik orbitope and for points with arclength above this threshold, their lifts do not form an edge. They also give a lower bound for this threshold and conjecture that this bound is tight. Results of Smilansky prove tightness for $k=2$. Here we prove this conjecture for all $k$. 
\end{abstract}

\section{The odd trigonometric and cosine moment curves}
Understanding the facial structure of the convex hull of curves is critical to the study of convex bodies, such as orbitopes and spectrahedron. 
It also reveals faces of polytopes formed by taking the convex hull of finitely many points on the curve. 
In 2008, Barvinok and Novik \cite{BN} use this technique to derive new asymptotic lower bounds for the maximal face numbers of centrally symmetric polytopes. To do this they study the symmetric trigonometric moment curve and the faces of its convex hull. 
Following \cite{BN}, let $SM_{2k}$ denote the symmetric trigonometric moment curve,
\[SM_{2k}(\theta) = (\cos(\theta), \cos(3\theta), \hdots, \cos((2k-1)\theta),\sin(\theta),\sin(3\theta),\hdots, \sin((2k-1)\theta)),\] 
and $B_{2k}$ its convex hull,
\[B_{2k} = \conv(SM_{2k}([0,2\pi])).\] 
Barvinok and Novik show that $B_{2k}$ is locally $k$-neighborly and use this to produce centrally symmetric polytopes with high faces numbers. 
The convex body $B_{2k}$ is also an \textit{orbitope}, that is, the convex hull of the orbit of a compact group (e.g. $\mathbb{S}^1$) acting linearly on a vector space, as studied in \cite[\S 5]{orb}. It is also remarked that the convex hull of the full trigonometric moment curve is the Hermitian Toeplitz spectrahedron, meaning that $B_{2k}$ is the projection of this Toeplitz spectrahedron \cite{orb}. For example, 
\[B_4 = \left\{ (x_1,x_3,y_1,y_3) \in \R^4\;:\; \exists \;z_2 \in \mathbb{C} \text{ with }\begin{bmatrix} 1&z_1&z_2&z_3 \\ \overline{z_1}&1&z_1&z_2\\\overline{z_2}&\overline{z_1}&1&z_1\\ \overline{z_3}&\overline{z_2}&\overline{z_1}&1 \end{bmatrix} \succeq 0\right\}\]
where $z_j = x_j+i y_j$ and ``$M \succeq 0$" denotes that the Hermitian matrix $M$ is positive semidefinite. Smilansky \cite{Smil} studies in depth the convex hulls of four-dimensional moment curves, such as $B_4$, and completely characterizes their facial structure. 

As an orbitope, the projection of a spectrahedron, and convex hull of a curve, the centrally symmetric convex body $B_{2k}$ is an interesting object in its own right, in addition to its ability to provide centrally symmetric polytopes with many faces. The theorem of this paper is a complete characterization of the edges of $B_{2k}$, which gives an affirmative answer to the first question of \cite[Section 7.4]{BN}. 
 
\begin{thm}
For $\alpha \neq \beta \in [0,2\pi]$, the line segment $[SM_{2k}(\alpha),SM_{2k}(\beta)]$ is 
\begin{center}
\begin{tabular}{ll}an exposed edge of $B_{2k}$& if $|\alpha - \beta| < 2\pi(k-1)/(2k-1)$, and\\
 not an edge of $B_{2k}$& if $|\alpha - \beta| > 2\pi(k-1)/(2k-1)$,\end{tabular}\end{center}
where $|\alpha-\beta|$ is the length of the arc between $e^{i\alpha}$ and $e^{i\beta}$ on $\bS^1$. \label{edge}\end{thm}

Our contribution is to prove the second case, when $[SM_{2k}(\alpha),SM_{2k}(\beta)]$ is not an edge. The existence of exposed edges is given by the following:
\begin{thm}[{\cite[Theorem 1.1]{BN}}] For all $k \in \Z_{>0}$, there exists $\frac{2\pi(k-1)}{2k-1}\leq\psi_k~\leq~\pi$ so that for all $\alpha \neq \beta \in [0,2\pi]$, the line segment $[SM_{2k}(\alpha),SM_{2k}(\beta)]$ is an exposed edge of $B_{2k}$ if $|\alpha - \beta| < \psi_k$ and not an edge of $B_{2k}$ if $|\alpha - \beta| > \psi_k$.\end{thm}

To prove Theorem \ref{edge}, it suffices to show that for arbitrarily small $\epsilon >0$ and $|\alpha - \beta| = 2\pi(k-1)/(2k-1)+\epsilon$,  the line segment $[SM_{2k}(\alpha),SM_{2k}(\beta)]$ is not an edge of $B_{2k}$. By the $\mathbb{S}^1$ action on $B_{2k}$, $[SM_{2k}(\alpha),SM_{2k}(\beta)]$ is an edge of $B_{2k}$ if and only if $[SM_{2k}(\alpha+\tau),SM_{2k}(\beta+\tau)]$ is an edge for all $\tau \in [0,2\pi]$. Thus is it suffices to show that $[SM_{2k}(-\theta),SM_{2k}(\theta)]$ is not an edge of $B_{2k}$ for $\theta = \pi(k-1)/(2k-1)+\epsilon/2$.

To study $SM_{2k}$ we will look at the projection onto its ``cosine components". Let \[C_k(\theta) = (\;\cos(\theta), \cos(3\theta),  \hdots, \cos((2k-1)\theta)\;) \;\subset\; \mathbb{R}^k.\]
By (\ref{eq:midpoint}) below, $C_k$ is the curve of midpoints of the line segments $[SM_{2k}(-\theta), SM_{2k}(\theta)]$. 

\begin{lemma} If $C_k(\theta)$ lies in the interior of $\conv(C_k)$, then $[SM_{2k}(-\theta),SM_{2k}(\theta)]$ is \textit{not} an edge of $B_{2k}$. \label{cos}
\end{lemma}
\begin{proof}
Let  $L = \{x\in \R^{2k}\;:\; x_{k+1}=\hdots=x_{2k}=0\}$. Note that for all $\theta \in [0,2\pi]$, $L\cap B_{2k}$ contains the point 
\begin{equation}(C_k(\theta),\overline{0})   \;= \;\frac{1}{2}SM_{2k}(-\theta)+\frac{1}{2}SM_{2k}(\theta), \label{eq:midpoint}\end{equation} and the convex hull of these points is full-dimensional in $L$. As $L$ contains the point $(0,\hdots,0)$, it intersects the interior of $B_{2k}$. Thus the relative interior of $B_{2k}\cap L$ and the intersection of $L$ with the interior of $B_{2k}$ coincide.

By assumption, $C_k(\theta)$ lies in the interior of $\conv(C_k)$, meaning that the point $\frac{1}{2}SM_{2k}(-\theta)+\frac{1}{2}SM_{2k}(\theta)$ lies in the relative interior of $L \cap B_{2k}$. Thus the line segment $[SM_{2k}(-\theta), SM_{2k}(\theta)]$ intersects the interior of $B_{2k}$ and it cannot be an edge. \end{proof}

To prove Theorem \ref{edge}, it now suffices to show that for small enough $\epsilon>0$, $C_k( \frac{k-1}{2k-1}\pi +\epsilon)$ lies in the interior of $\conv(C_k)$. It will be worth noting that $\cos( d \theta)$ is a polynomial of degree $d$ in $\cos(\theta)$, called the $d$th Chebyshev polynomial \cite{R}. Thus $C_k$ is a segment of an algebraic curve of degree $2k-1$, parametrized by the Chebyshev polynomials of odd degree evaluated in $[-1,1]$. 

\section{Curves dipping behind facets}

Here we give a criterion for a curve $C$ to dip inside of its convex hull after meeting a facet of $\conv(C)$.
Let $C(t) = (C^1(t), \hdots, C^n(t))$,  $t \in [-1,1]$ be a curve in $\R^n$ where $C^i \in \R[t]$. 
Let $F$ be a facet of $\conv(C)$ with supporting hyperplane $\{h^Tx = h_0\}$.  Suppose $C(t_0)$ is a vertex of $F$ with $t_0\in (-1,1)$ and that $C$ is smooth this point (\textit{i.e.} $C'(t_0) \neq \overline{0}$). Let $\pi_F$ denote the projection of $\R^n$ on to the affine span of $F$. See Figure~\ref{faceproj} for an example.

\begin{lemma} If $\pi_F(C(t_0+\epsilon))$ lies in the relative interior of $F$ for small enough $\epsilon >0$ and  any facet of $F$ containing $C(t_0)$ meets the curve $\pi_F(C)$ transversely at this point, then $C(t_0+\epsilon)$ lies in the interior of $\conv(C)$. \label{face}
 \end{lemma}
\begin{proof}
Let $p$ be a point on $C \backslash F$. Then $\conv(F  \cup p)$ is a pyramid over the facet $F$. We will show that $C(t_0+\epsilon)$ lies in the interior of this polytope. Suppose $\{h^Tx\leq h_0,\;a_i^Tx\leq b_i, \;i=1,\hdots, s\}$ is a minimal facet description of $\conv(F  \cup p)$ with $a_i\in \R^n$, $b_i\in \R$. 
Then $a_i^T x<b_i$ for all $x$ in the relative interior of $F$.  

The polynomial $h_0-h^TC(t)\in \R[t]$ is non-negative for all $t\in [-1,1]$. As this polynomial is non-zero, it has only finitely many roots. Thus, for small enough $\epsilon >0$, $h^TC(t_0+\epsilon)<h_0$. 

Now we show that $a_i^TC(t_0+\epsilon)<b_i$. As $h_0-h^TC(t)$ is non-negative and zero at $t_0\in (-1,1)$, it must have a double root at $t_0$. This implies that $h^TC'(t_0)=0$, and thus, for any $\epsilon$, the point $C(t_0)+\epsilon C'(t_0)$ lies in the affine span of $F$. 
As $C(t_0)$ and $C(t_0)+\epsilon C'(t_0)$ both lie in the affine span of $F$, we have that 
\begin{align}
a_i^TC(t_0+\epsilon)\; &=\;a_i^TC(t_0)+\epsilon a_i^TC'(t_0) + O(\epsilon ^2),  \;\;\text{ and } \label{eq:curve1}\\
a_i^T\pi_F(C(t_0+\epsilon))\; &=\;a_i^TC(t_0)+\epsilon a_i^TC'(t_0) + O(\epsilon ^2).\label{eq:curve2} \end{align}
Our transversality assumption implies that, for each $i=1, \hdots s$, if $a_i^TC(t_0)=b_i$ then $a_i^T\pi_F(C'(t_0))=a_i^TC'(t_0)\neq 0$.  
Then for small enough $\epsilon>0$, $a_i^TC(t_0)+\epsilon a_i^TC'(t_0)$ is non-zero. As $\pi_F(C(t_0+\epsilon))$ lies in the relative interior of $F$,  $a_i^T\pi_FC(t_0+\epsilon)<b_i$. 
By (\ref{eq:curve2}), this implies that $a_i^TC(t_0)+\epsilon a_i^TC'(t_0) <b_i$. It then follows from (\ref{eq:curve1}) that $a_i^TC(t_0+\epsilon)<b_i$. 

This shows that $C(t_0 +\epsilon)$ lies in the interior of $\conv(F \cup p) \subset \conv(C)$. 
\end{proof}

\begin{figure}
\begin{center}
\includegraphics[scale=0.5, trim=50 470 250 120]{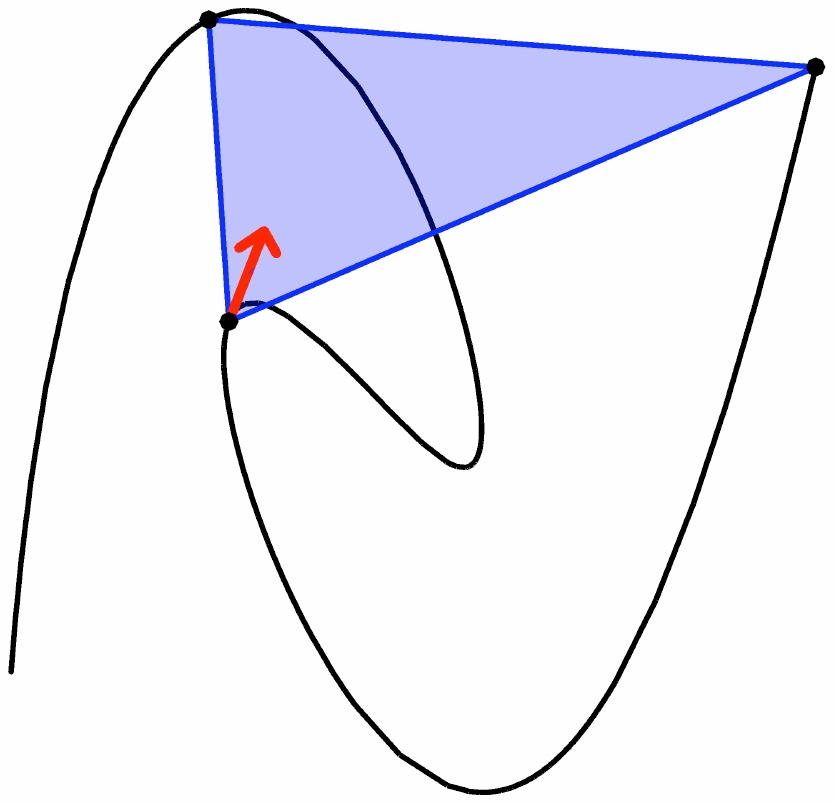}$\;\;\; \begin{picture}(35,90) \put(5,70){\vector(1,0){30}} \put(13,80){$\pi_F$} \end{picture} \;\;\;$ \includegraphics[scale=0.45, trim=50 470 250 120]{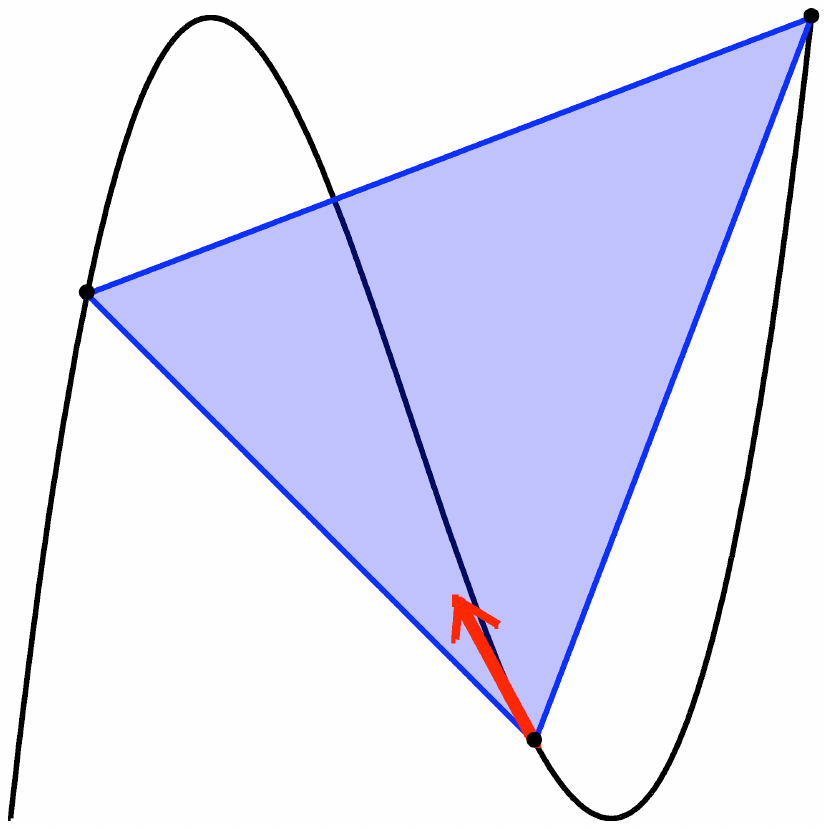}\end{center} \caption{Projection of the curve $C_3$ onto the facet $\{x_3=1\}$ of its convex hull. The tangent vector $C_3(t_0)+C_3'(t_0)$ for $t_0=2\pi/5$ is shown in red.}\label{faceproj}\end{figure}

\begin{remark}\label{LP} The hypotheses of Lemma~\ref{face} are equivalent to the condition that for small $\epsilon>0$, $C(t_0)+\epsilon C'(t_0)$ lies in the relative interior of $F$, or rather, that the vector $C'(t_0)$ lies in the relative interior of the tangent cone of $F$ at $C(t_0)$.  Given $F$, $C(t_0)$, and $C'(t_0)$, checking this condition is a linear program. 
\end{remark}


\section{Understanding the facet $\{x_k=1\}$}

We will show that the hypotheses of Lemma~\ref{face} are satisfied using the curve $C=C_k$, facet $F = \{x_k=1\}\cap\conv(C_k)$, and point $C(t_0)=C_k(\frac{k-1}{2k-1}\pi )$. To do this, we have to understand this facet and the projection of $C_k$ onto the hyperplane $\{x_k=1\}$. 

Note that the intersection of $C_k$ with the hyperplane $\{x_k=1\}$ is $k$ points given by solutions to $\cos((2k-1)\theta)=1$ in $[0,\pi]$, namely $\{C_k( \frac{2j}{2k-1}\pi) \;:\; j=0, \hdots, k-1\}$. The projection of $C_k$ onto this hyperplane is just $(C_{k-1},1)$.  Thus to understand the projection of $C_k$ onto this facet, we need to look at the points $\{ C_{k-1}( \frac{2j}{2k-1}\pi) \;:\; j=0, \hdots, k-1\}$. 
Let
\[\theta_0 = \frac{\pi}{2} \;\;\;\;\;\;\; \text{   and   }\;\;\;\;\;\;\; \theta_j = \frac{2j}{2k-1}\pi\;\;\;\;\;\; \text{  for   }j=1, \hdots, k-1.\]
Define the following two polytopes (simplices) in $\R^{k-1}$:
\begin{align*}
P_k &= \conv(\{C_{k-1}(0 \pi)\}\cup\;\{C_{k-1}( \theta_j ) \;:\; j=1, \hdots, k-1\})\\ &\\
Q_k &=  \conv(\{C_{k-1}( \theta_j ) \;:\; j=0, \hdots, k-1\}).\end{align*}

\begin{figure}\begin{center}\includegraphics[scale=0.6, trim=0 20 0 50]{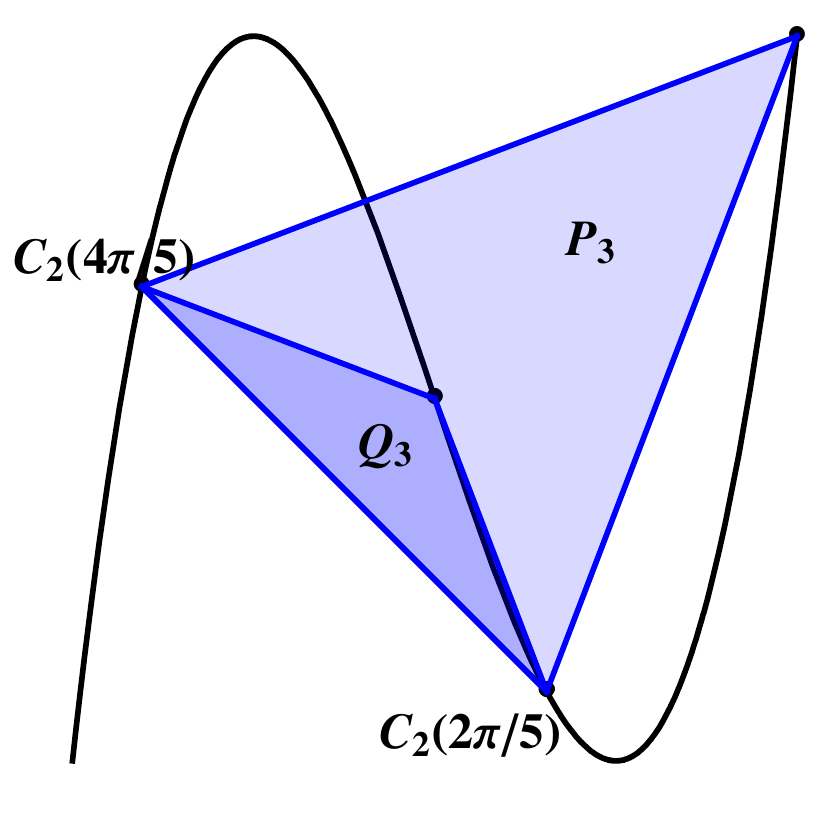}$\;\;\;\;\;\;\;\;\;\;\;\;\;\;\;$ \includegraphics[scale=0.6,trim=10 70 50 60]{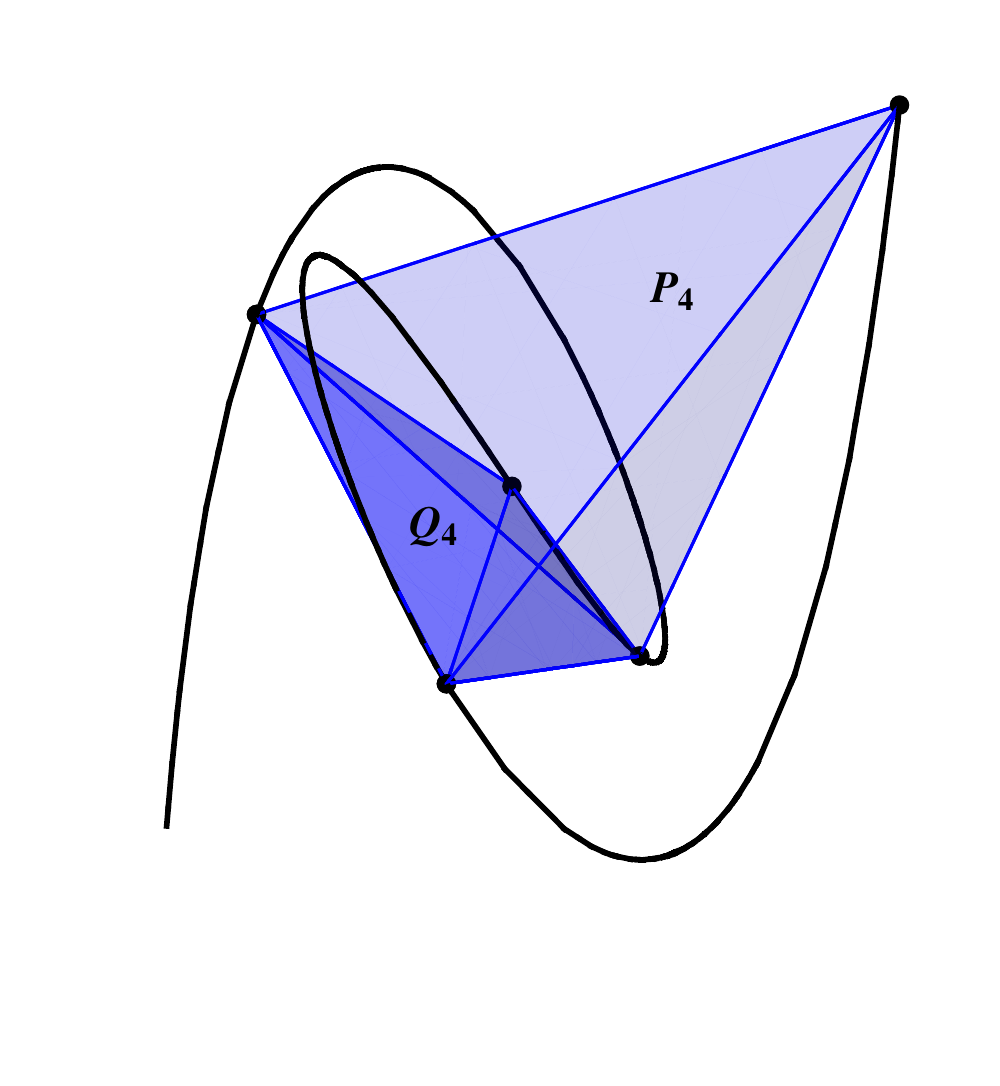} \end{center}\caption{ On the left, the curve $C_2$ (the projection of $C_3$ onto the plane $\{x_3=1\}$) with the triangles $P_3$ and $Q_3$. On the right, $C_3$ with the tetrahedra $P_4$ and $Q_4$.}\end{figure}
While $P_k$ is the polytope we'll use as $F$ in Lemma \ref{face}, $Q_k$ is a simplex which sits inside of $P_k$ and has a more tractable facet description. We will show that  $C_{k-1}(\frac{k-1}{2k-1}\pi +\epsilon)$ lies in $Q_k$ in order to show that it lies in $P_k$. We'll often need the trigonometric identities stated in Section \ref{trig}.

To see that $Q_k \subseteq P_k$, note that their vertex sets differ by only one element. It suffices to write $Q_k$'s extra vertex, $(0,\hdots, 0)=C_{k-1}(\frac{\pi}{2})$, as a convex combination of the vertices of $P_k$. 
By Trig. Identity \ref{trig sum 2}, we have that for each $l=1, \hdots, k-1$, $0=1/2 + \sum_{j=1}^{k-1} \cos((2l-1)\theta_j)).$ Putting these together gives that $C_{k-1}(\frac{\pi}{2})=(0,\hdots,0) =\frac{2}{2k-1} (\frac{1}{2}C_{k-1}(0\pi) + \sum_{j=1}^{k-1}C_{k-1}(\theta_j))$. So indeed $Q_k\subset P_k$.

\begin{lemma} The curve $C_{k-1}$ meets each facet of $Q_k$ transversely and $C_{k-1}(\theta)$ lies in the interior of $Q_k \subset P_k$ for
$\theta  \in  
\left\{ \begin{array}{rl}
(\frac{(k-1)\pi}{2k-1},\frac{\pi}{2}) & \text{if } k \text{ is odd}\\
(\frac{\pi}{2},\frac{k\pi}{2k-1}) & \text{if } k \text{ is even.}\\
\end{array}\right .$
\label{Qk}
\end{lemma}

\begin{proof} The plan is to find a halfspace description of $Q_k$, find the places where $C_{k-1}$ crosses the boundary of each of these halfspaces, and deduce from this that $C_{k-1}(\theta)$ lies in each of these halfspaces for the appropriate $\theta$.  

First we find the facet description of $Q_k$.
For $k\in \N$, and $j\in \{0, \hdots, k-1\}$, define the affine linear functions $h_{j,k}:\R^{k-1} \rightarrow \R$ as 
\begin{align*} h_{0,k}(x) &= 1/2+ \sum_{l=1}^{k-1}x_l, \;\;\;\;\;\text{and}  \\ 
h_{j,k}(x) &= \sum_{l=1}^{k-1}\left(\cos((2l-1)\theta_j)-1\right)\; x_l \;\;\;\;\;\text{ for }j=1, \hdots, k-1.\end{align*}
We will see that $Q_k = \{x \in \R^{k-1}\;:\; h_{j,k}(x)\geq 0 \;\text{ for all } j=0, \hdots, k-1\}$. 
Note that each $h_{j,k}$ gives a trigonometric polynomial by composition with $C_{k-1}$. For each $j=0, \hdots, k-1$, define $f_{j,k}:[0,2\pi]\rightarrow \R$ by \[f_{j,k}(\theta) := h_{j,k}(C_{k-1}(\theta)).\] 
To see that the $h_{j,k}$ give a facet description of $Q_k$ we will show that for each $j=0, \hdots, k-1$, we have $f_{j,k}(\theta_j)>0$ and $f_{j,k}(\theta_i)=0$ for all $i \neq j$. 
By Trig. Identity \ref{trig sum} in Section \ref{trig}, \[f_{0,k}(\theta_j) \;=\;\frac{1}{2}+ \sum_{l=1}^{k-1}\cos((2l-1)\theta_j) \;=\;0  \]
for $ j=1, \hdots, k-1$. Moreover $f_{0,k}(\theta_0) =f_{0,k}(\frac{\pi}{2}) =  1/2+\sum_{l=1}^{k-1} 0 >0.$\\$\;$\\
Now let $j \in\{1,\hdots, k-1\}$. Using Trig. Identities \ref{trig sum} and \ref{trig prod sum}, we see that for every $i \in \{1, \hdots, k-1\} \backslash \{j\}$,
\begin{align*}f_{j,k}(\theta_i) &= \sum_{l=1}^{k-1} \cos((2l-1)\theta_j)\cos((2l-1)\theta_i) -  \sum_{l=1}^{k-1}\cos((2l-1)\theta_i)\\
& = -\frac{1}{2} - (-\frac{1}{2}) \;=\;0. \end{align*}
Also, we have $f_{j,k}(\theta_0) =f_{j,k}(\frac{\pi}{2}) = h_{j,k}(\overline{0}) = 0$. Finally \begin{align*}f_{j,k}(\theta_j) &=  \sum_{l=1}^{k-1} \cos((2l-1)\theta_j)^2 -  \sum_{l=1}^{k-1}\cos((2l-1)\theta_j)\\
&=\sum_{l=1}^{k-1} \cos((2l-1)\theta_j)^2 +\frac{1}{2}\tag{by Trig. Identity \ref{trig sum}} \\
& > 0. \end{align*}
So indeed $Q_k = \{x\in \R^{k-1} \;:\; h_{j,k}(x)\geq 0 \;\text{ for all } j=0, \hdots, k-1\}$.

To prove Lemma \ref{Qk}, it suffices to show that all roots of $f_{j,k}$ have multiplicity one and $f_{j,k}(\theta) >0$ for the specified $\theta$.  We start by finding all roots of  $f_{j,k}(\theta)$ in $[0,\pi]$. 

\begin{remark}As $C_d$ is an algebraic curve of degree $2d-1$ in $\cos(\theta)$, it meets any hyperplane in at most $2d-1$ points (counted with multiplicity). 
\end{remark}
Thus for each $j$, $f_{j,k}$ has at most  $2k-3$ roots in $[0,\pi]$. We have already found $k-1$ roots of each, namely $\{\theta_0, \hdots, \theta_{k-1}\}\backslash \{\theta_j\}$. Now we find the remaining $k-2$.\\ $\;$\\
\textbf{(j=0).} Note that $\cos(\pi -\theta)  = -\cos(\theta)$. Then by Trig. Identity \ref{trig sum}, for $i=1, \hdots, k-2$,  
\begin{align*} f_{0,k}\left(\frac{2i-1}{2k-3}\pi\right) 
=\;\;&\sum_{l=1}^{k-1}\cos\left(\frac{(2l-1)(2i-1)}{2k-3}\pi\right)+ \frac{1}{2}\\ 
=-1+&\sum_{l=1}^{k-2}\cos\left(\frac{(2l-1)(2i-1)}{2k-3}\pi\right)+ \frac{1}{2}\; =\;-1 + \frac{1}{2}+ \frac{1}{2}\;=\;0.\end{align*}

Thus the roots of $f_{0,k}$ are $\{\theta_i \;:\; i=1, \hdots, k-1\} \cup \{\frac{(2i-1)\pi}{2k-3} \;:\; i=1, \hdots, k-2\}$. As there are $2k-3$ of them, we know that these are all the roots of $f_{0,k}$ and each occurs with multiplicity one. Furthermore, since
\[\frac{k-2}{2k-3}\;<\;\frac{k-1}{2k-1}\;<\;\frac{k}{2k-1}\;<\;\frac{k-1}{2k-3},\]
it follows that $f_{0,k}$ has no roots in the interval $(\frac{(k-1)\pi}{2k-1}, \frac{k\pi}{2k-1})$.   Thus the sign of $f_{0,k}$ is constant on $(\frac{(k-1)\pi}{2k-1}, \frac{k\pi}{2k-1})$. Since $f_{0,k}(\frac{\pi}{2}) >0$, we see that for all $\theta \in  (\frac{(k-1)\pi}{2k-1}, \frac{k\pi}{2k-1})$,  $f_{0,k}(\theta)=h_{0,k}(C_{k-1}(\theta))>0$.\\$\;$\\
\textbf{(j =1, $\hdots$, k-1).} Note that $f_{j,k}(\pi - \theta) = -f_{j,k}(\theta)$. We've already seen that $\theta_i = \frac{2i \pi}{2k-1}$ is a root of this function for $i \in \{1, \hdots, k-1\} \backslash \{j\}$, so for each such $i$, $\frac{(2k-1-2i)\pi}{2k-1}$ is also a root. Thus the $2k-3$ roots of $f_{j,k}( \theta)$  are \[\left\{\frac{\pi}{2}\right\} \cup\left\{\frac{i \pi}{2k-1} \;:\; i \in \{1, \hdots, 2k-2\}\backslash\{2j, 2k-1-2j\}\right\}.\]  

For each $j$ this gives that $f_{j,k}$ has $k-1$ roots of multiplicity one in $[0, \frac{(k-1)\pi}{2k-1}]$ and no roots in $(\frac{(k-1)\pi}{2k-1}, \frac{\pi}{2})$. Note that $f_{j,k}(0\pi) <0$. The sign of $f_{j,k}(\theta)$ changes at each of its roots, so  for $\theta \in (\frac{(k-1)\pi}{2k-1}, \frac{\pi}{2})$, we have that $(-1)^{k-1}f_{j,k}(\theta) > 0$. By symmetry of $ f_{j,k}(\theta)$ over $\pi /2$, we see that for $\theta \in (\frac{\pi}{2}, \frac{k\pi}{2k-1})$ we have $(-1)^{k}f_{j,k}(\theta) > 0$. 
\end{proof}

\begin{figure}[h]\begin{center}\includegraphics[scale=0.7]{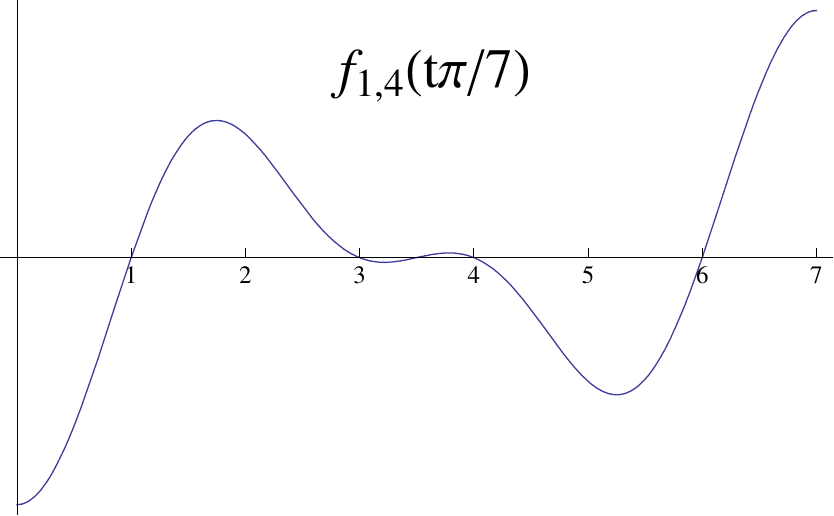}$\;\;\;\;\;\;\;\;\;\;\;\;\;\;\;$ \includegraphics[scale=0.7]{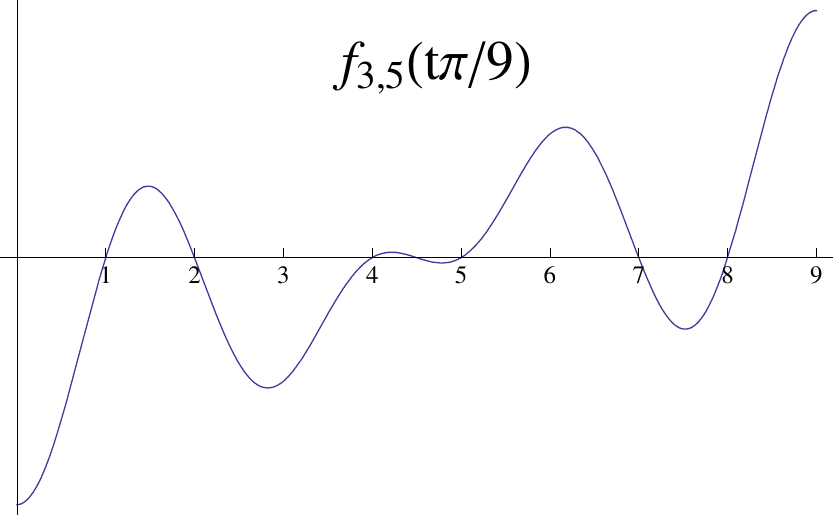}\end{center}\caption{Here are two examples of the graphs of $f_{j,k}(\theta)$. Note that $f_{j,k}(\frac{\pi}{2k-1}t)$ has roots $\{1,\hdots, 2k-1\} \backslash\{2j, 2k-1-2j\}$, all of multiplicity one.} \end{figure}

Now that we completely understand the facets of $Q_k$ and their intersection with the curve $C_{k-1}$, we can use the previous lemmata to prove our main theorem.

\section{Proof of Theorem \ref{edge}}
\begin{proof}
As discussed before, by \cite[Thm 1.1]{BN} and symmetry of the faces it suffices to show that  for arbitrarily small $\epsilon >0$ and $\theta = \frac{k-1}{2k-1}\pi+\epsilon$, $[SM_{2k}(-\theta), SM_{2k}(\theta)]$ is not an edge of $B_{2k}$. By Lemma \ref{cos}, we can do this by showing that $C_k(\frac{k-1}{2k-1}\pi+\epsilon)$ lies in the interior of $\conv(C_k)$. 

Note that $C_k(\frac{k-1}{2k-1}\pi+\epsilon)$ lies in the interior of $\conv(C_k)$ if and only if $C_k(\frac{k}{2k-1}\pi-\epsilon)$ lies in the interior of $\conv(C_k)$. 
As the value of $\cos((k-1)\pi)$ depends on the parity of $k$, we will use  $C_k(\frac{k-1}{2k-1}\pi+\epsilon)$ for odd $k$ and $C_k(\frac{k}{2k-1}\pi-\epsilon)$ for even $k$.
 
 We know that $\conv(C_k)$ has a face given by $x_{k}=1$. This intersects $C_k$ at the points $\{C_k(0\pi)\} \cup \{ C_k(\theta_j) \;:\; j=1,\hdots, k-1\}$.  Thus, the intersection of $\conv{C_k}$ with $\{x_k=1\}$ is $P_k$ as defined earlier sitting at height 1, and the projection of $C_k$ onto $\{x_k=1\}$ is $C_{k-1}$.\\$\;$\\ 
 \textbf{k odd.}  Since $k-1$ is even,  $C_k(\frac{k-1}{2k-1}\pi)$ lies on the face defined by $x_k =1$. Moreover, for small enough $\epsilon >0$, $C_{k-1}(\frac{k-1}{2k-1}\pi +\epsilon)$ is in the interior of $Q_k\subset P_k$ by Lemma~\ref{Qk}. As the curve $C_{k-1}$ meets the facets of $Q_k$ transversely at $C_{k-1}(\frac{k-1}{2k-1}\pi)$, it must meet the facets of $P_k$ transversely at this point as well (see Remark~\ref{LP}).   Lemma~\ref{face} then shows that $C_k(\frac{k-1}{2k-1}\pi + \epsilon)$ lies in the interior of $\conv(C_k)$ for small enough $\epsilon >0$.  \\$\;$\\ 
 \textbf{k even.}  Now  $k$ is even and $C_k(\frac{k}{2k-1}\pi)$ lies on the face defined by $x_k =1$. As before, for small enough $\epsilon >0$, $C_{k-1}(\frac{k}{2k-1}\pi -\epsilon)$ is in the interior of $P_k$ and $C_{k-1}$ meets the facets of $P_k$ transversely at $C_{k-1}(\frac{k}{2k-1}\pi)$. Thus $C_k(\frac{k}{2k-1}\pi - \epsilon)$ lies in the interior of $\conv(C_k)$ for small enough $\epsilon >0$.
 \end{proof}

We now know all the edges of $B_{2k}$.  This leaves the challenging open problem of understanding the higher dimensional faces of this convex body.

\section{Useful trigonometric identities}\label{trig}

\begin{trig}For any $k \in \N$ and $l\in\{1, \hdots, k-1\}$,  \[ \sum_{j=1}^{k-1} \cos\left(\frac{(2l-1)2j}{2k-1}\pi\right)=-\frac{1}{2}.\]\label{trig sum 2}
\end{trig} \begin{proof}
By \cite[Ex. 1.5.26]{R}, for $l=1, \hdots, k-1$, we have that $0 = 1+ \sum_{j=1}^{2k-2}\cos\left( \frac{(2l-1)j}{2k-1}\pi\right).$

As $-j\equiv 2k-1-j\mod 2k-1$ and $\cos(\theta)=\cos(-\theta)$, this gives
\begin{align*}0&=1+ \sum_{j=1}^{2k-2}\cos\left( \frac{(2l-1)j}{2k-1}2\pi\right) \\
&=1+\sum_{j=1}^{k-1}\left[\cos\left( \frac{(2l-1)j}{2k-1}2\pi\right) +\cos\left( \frac{(2l-1)(2k-1-j)}{2k-1}2\pi\right)\right]\\
&=1 +2\;\sum_{j=1}^{k-1}\cos\left( \frac{(2l-1)j}{2k-1}2\pi\right).
\end{align*}\end{proof}

\begin{trig} For any $k \in \N$ and $j\in\{1, \hdots, 2k-2\}$,  \[\sum_{l=1}^{k-1}\cos\left(\frac{(2l-1)2j}{2k-1}\pi\right)=-\frac{1}{2}. \]\label{trig sum}
\end{trig}
\begin{proof} 




By \cite[Ex. 1.5.26]{R}, we have that for $j=1, \hdots, 2k-2$, \[0= \sum_{l=1}^{2k-1}\cos\left(\frac{(2l-1)2j}{(2k-1)}\pi\right) = 1+ \sum_{l=1}^{2k-2}\cos\left(\frac{(2l-1)2j}{(2k-1)}\pi\right).\] 
From this, the claim follows by an argument similar to the proof of Trig. Identity~\ref{trig sum 2}. \end{proof}

\begin{trig} For any $k \in \N$ and $i\neq j \in  \{0, \hdots, k-1\}$,  \[\sum_{l=1}^{k-1}\cos\left(\frac{(2l-1)2i}{2k-1}\pi\right)\cos\left(\frac{(2l-1)2j}{2k-1}\pi\right)=-\frac{1}{2}.\] \label{trig prod sum}\end{trig}
\begin{proof} As $|i-j|,|i+j| \in \{1,\hdots, 2k-2\}$, this follows from Trig. Identity \ref{trig sum} and the identity $\cos(\alpha)\cos(\beta) = \frac{1}{2}\cos(\alpha+\beta)+\frac{1}{2}\cos(\alpha-\beta)$. \end{proof} 


\section*{Acknowledgements}Thanks to Ming Xiao Li and Raman Sanyal for many helpful discussions and to the reviewers for their careful reading. 
 The author was funded by the University of California - Berkeley Mentored Research Award and
NSF grant DMS-0757207.

\end{document}